\documentclass[12pt]{article}
\usepackage{amsfonts,amssymb,amscd,amsmath,enumerate,verbatim,calc,graphicx}
\usepackage[colorlinks,breaklinks,backref]{hyperref}\usepackage{amsthm}
\usepackage{varwidth}
\usepackage{algpseudocode}
\usepackage[]{color}\usepackage{color}
\usepackage{tikz}
\usetikzlibrary{arrows}
\newtheorem{thm}{Theorem}[section]
\newtheorem{cor}[thm]{Corollary}
\newtheorem{lem}[thm]{Lemma}

\theoremstyle{definition}
\newtheorem{defn}[thm]{Definition}
\theoremstyle{remark}

\newcommand{\A}{\mathcal{A}}
\newcommand{\aut}{{\rm Aut}}
\newcommand{\B}{\vec{\mathcal B}}

\newcommand{\id}{\textsf{ID}_t^S}
\begin{document}

\title{Identifying Codes on Directed De Bruijn Graphs}%
\author{Debra Boutin \thanks{\texttt{dboutin@hamilton.edu}} \\ Department of Mathematics \\ Hamilton College \\ \\ Victoria Horan Goliber \thanks{\texttt{victoria.goliber@us.af.mil}} \\  Air Force Research Laboratory \\ Information Directorate \\ \\ Mikko Pelto \thanks{\texttt{mikko.pelto@utu.fi}, research supported by the Finnish Cultural Foundation} \\ Department of Mathematics and Statistics \\ University of Turku}

\date{\today}%

\maketitle

\let\thefootnote\relax\footnote{Approved for public release; distribution unlimited:  88ABW-2015-3797.}

\begin{abstract}
For a directed graph $G$, a {\it $t$-identifying code} is a subset $S\subseteq V(G)$ with the property that for each vertex $v\in V(G)$ the set of vertices of $S$ reachable from $v$ by a directed path of length at most $t$ is both non-empty and unique. A graph is called {\it $t$-identifiable} if there exists a $t$-identifying code.  This paper shows that the de~Bruijn graph $\B(d,n)$ is $t$-identifiable if and only if $n \geq 2t-1$.  It is also shown that a $t$-identifying code for  $t$-identifiable de~Bruijn graphs must contain at least $d^{n-1}(d-1)$ vertices, and constructions are given to show that this lower bound is achievable $n \geq 2t$.  Further a (possibly) non-optimal construction is given when $n=2t-1$.  Additionally, with respect to $\vec{\mathcal{B}}(d,n)$ we provide upper and lower bounds on the size of a minimum \textit{$t$-dominating set} (a subset with the property that every vertex is at distance at most $t$ from the subset), that the minimum size of a \textit{directed resolving set} (a subset with the property that every vertex of the graph can be distinguished by its directed distances to vertices of $S$) is $d^{n-1}(d-1)$, and that if $d>n$ the minimum size of a {\it determining set} (a subset $S$ with the property that the only automorphism that fixes $S$ pointwise is the trivial automorphism) is $\left\lceil \frac{d-1}{n}\right\rceil$.

\end{abstract}

\section{Introduction}

First introduced in 1998 \cite{original}, an {\it identifying code} for an undirected graph $G$ is a subset  $S\subseteq V(G)$  with the property that for each $v\in V(G)$ the subset of vertices of $S$ that are adjacent to $v$ is non-empty and unique.  That is, each vertex of the graph is uniquely identifiable by the non-empty subset of vertices of $S$ to which it is adjacent.  Note that not all graphs have an identifying code; those that do are called {\it identifiable}.  A graph fails to be identifiable if and only if it contains a pair of  vertices with the same closed neighborhood; such vertices are called {\it twins}. Extending these definitions to $t$-identifying and $t$-twins in which we consider $t$-neighborhoods is easy and is covered in Section \ref{IDSet}. Identifying codes can be quite useful in applications. For example, an identifying code in a network of smoke detectors allows us to determine the exact location of a fire given only the set of detectors that have been triggered.  However, the problem of finding identifying codes is  NP-Hard \cite{NPHard}.  The computational cost has so far limited the real-world use of identifying codes.\medskip

The directed de~Bruijn graph $\B(d,n)$ is a directed graph in which the vertices  are strings of length $n$ from an alphabet $\A$ with $d$ letters, and with a directed arc from vertex $x_1x_2\ldots x_n$ to vertex $x_2\ldots x_n a$ for each $a\in \A$.  When looking for a graph model for applications, it is useful to choose a graph with relatively few edges, but many short paths between any pair of vertices \cite{Baker}.  The de~Bruijn graphs have both of these desirable properties. In addition, given an arbitrary pair of vertices in a de~Bruijn graph, there are routing algorithms that, with high probability, create a path of length $O(\log n)$ between the pair \cite{Richa}.  The properties of de~Bruijn graphs enable some problems that are NP-complete on general graphs, such as the Hamilton cycle problem, to be computationally solvable on de~Bruijn graphs \cite{Skiena}.  We will consider identifying codes on directed graphs, and see that for most directed de~Bruijn graphs, the construction of minimum $t$-identifying codes is indeed solvable (when they exist).\medskip

Other vertex subsets that we consider for de~Bruijn graphs in this paper are dominating sets, resolving sets, and determining sets.  Dominating sets provide complete coverage for a graph, while resolving sets provide an identification of vertices in graphs using relative distances. Finally, determining sets provide a set of vertices that is only fixed pointwise by the trivial automorphism.  These types of subsets are also useful in applications.  For example, resolving sets have been used in aiding the navigation of robots when distances to sufficient landmarks are known \cite{Robots}, and determining sets are useful in graph distinguishing which can reduce graph symmetry to enhance recognition.  These different vertex subsets are interrelated, as shown in Figure \ref{relatives}.  For example, each resolving set and identifying code is also a determining set, but not vice-versa.  This is discussed more fully in Section \ref{Dominating} and  Section \ref{Determining}.\medskip

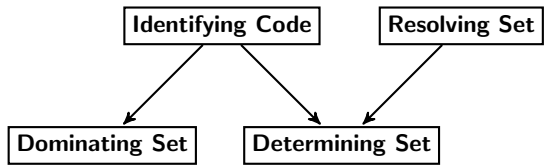
\begin{figure}
\begin{center}

\begin{tikzpicture}[->,>=stealth',shorten >=1pt,auto,node distance=3cm,
  thick,main node/.style={rectangle,draw,font=\sffamily\bfseries,scale=0.75},new node/.style={circle,fill=black,text=white,draw,font=\sffamily\bfseries,scale=0.75}]

  \node[main node] (0)                      {Identifying Code};
  \node[main node] (1) [below left of=0]    {Dominating Set};
  \node[main node] (2) [below right of=0]   {Determining Set};
  \node[main node] (3) [above right of=2]   {Resolving Set};

  \path[every node/.style={font=\sffamily\small}]
    (0) edge node [left]      {} (1)
        edge node [left]      {} (2)
    (3) edge node [left]      {} (2);

\end{tikzpicture}

\end{center}
\caption{Vertex Subset Relationships}\label{relatives}
\end{figure}

In Section \ref{definitions}, we give careful definitions necessary for working with directed de~Bruijn graphs.
In Section  \ref{IDSet}, 
we  prove that for every $t$-identifiable de~Bruijn graph, any $t$-identifying code must contain at least $d^{n-1}(d-1)$ vertices.  We prove by construction that this bound is achievable for $t$-identifying codes when $n \geq 2t$.  For  $n=2t-1$ we show the existence of a $t$-identifying code of size $d^{n-1}(d-1)+d^t$, which may or may not be optimal size.  Furthermore, when $n \leq 2t-2$, we are able to show that $\vec{\mathcal{B}}(d,n)$ has no $t$-identifying code due to the existence of \textit{$t$-twins}.  In Section \ref{OtherSet} we study dominating sets, directed resolving sets, and determining sets for $\B(d,n)$.  Section \ref{Resolving} gives a proof that  the minimum size of a directed resolving set for $\B(d,n)$ is $d^{n-1}(d-1)$, and Section \ref{Determining} that the minimum size of a determining set is $\left\lceil \frac{d-1}{n}\right\rceil$.\medskip

\section{Definitions}\label{definitions}

The bulk of our results depend on structures based on strings, and we require quite a few definitions and particular notations.  Let {\boldmath $\A_d$} be the set of integers modulo $d$ and let $\A_d^n$ be the set of all strings of length $n$ over $\A_d$.  When $d$ is clear from context we will use $\A$ and $\A^n$ respectively.  

Within these strings, repeated characters or strings may be denoted by an exponent.  For example, $$ z^i = \underbrace{zz \cdots z}_{i \text{ times}}$$ and $$(01)^k = \underbrace{(01)(01) \cdots (01)}_{k \text{ times}}.$$  Throughout this paper we also use the common notation $[d]=\{1, 2, \ldots , d\}$, as well as $[i,j]=\{i, i+1, \ldots , j-1, j\}$. \medskip

\begin{defn}
    The \textbf{concatenation of two strings} $v=v_1v_2 \ldots v_i$ and $w=w_1w_2 \ldots w_k$ is given by $v \oplus w = v_1v_2 \ldots v_iw_1w_2 \ldots w_k.$

    The \textbf{concatenation of sets of strings} $S$ and $T$ is given by $S \oplus T = \{v \oplus w \mid v \in S \hbox{ and } w \in T\}.$
\end{defn}

\begin{defn}
    When discussing substrings of a string $w=w_1w_2 \ldots w_n$, we will use the notation $w(a:b)$ to denote the substring $w_aw_{a+1} \ldots w_b$.  Two special substrings are the \textbf{prefix}, $w(1:n-1)$, and the \textbf{suffix}, $w(2:n)$.
\end{defn}

\begin{defn}
	Let $w = w_1w_2 \ldots w_n \in \A_d^n$.  Define the string operation $$w^{(i,a)} = w_1w_2 \ldots w_{i-1} a w_{i+1} \ldots w_n.$$  This operation simply replaces the $i$th letter in $w$ with $a$.
\end{defn}

\begin{defn}
    Let $w=w_1 \ldots w_n \in \A_d^n$ and $\ell \in \mathbb{Z}^+$ such that $n \geq 2 \ell$.  Then we say that $w$ has \textbf{period length $\ell$} if $w_i = w_{i+\ell}$ for all $i \in [n-\ell]$.  
    
    If we have $n < 2 \ell$, then we say that $w$ has \textbf{almost period length $\ell$}.
\end{defn}

\begin{defn}
    Let $w \in \A_d^n$, and suppose that $w$ has period length $\ell$, and does not have period length $k$ for any $k < \ell$.  Then $w$ is called \textbf{$\ell$-periodic}.

If there exist some $\ell > \frac{n}{2}$ and word $w' \in \A_d^{2\ell -n}$ such that $w \oplus w'$ is $\ell$-periodic, then $w$ is called \textbf{almost $\ell$-periodic}.
\end{defn}

We now provide some lemmas regarding string properties that will be used later in the paper.

\begin{lem}\label{MPL1}
    \emph{\cite{FineWilf}}  Let $\ell_1 > \ell_2$ and $w$ be a word of length $n \geq \ell_1+\ell_2 - \hbox{gcd}(\ell_1, \ell_2)$.  If $w$ has periods (or almost periods) of lengths $\ell_1$ and $\ell_2$ , then $w$ has a period of length gcd$(\ell_1, \ell_2)$.
\end{lem}

\begin{lem}\label{MPL2}
    \emph{\cite{Berstel}}  Let $\ell_1 \geq \ell_2$ and $w$ be a word of length $n \geq \ell_1+\ell_2$.  If $w$ has a period (or almost period) of length $\ell_1$ and $w^{(k,m)}$ has a period of length $\ell_2$ for some $m \in \A_d$, then there is $m'\in \A_d$ such that $w^{(k,m')}$ has a period of length gcd$(\ell_1, \ell_2)$.
\end{lem}

The next two lemmas provide methods for constructing non-periodic strings.

\begin{lem}\label{MPL3}
    Let $w \in \A_d^n$ such that $w$ is $\ell_1$-periodic or almost $\ell_1$-periodic.  Let $m \in [d-1]$ and also $k \in [n]$ with $k \leq n- \ell_1$ or $k > \ell_1$.   Then for any $\ell_2 < \frac{n}{2}$ with $\ell_1 \geq \ell_2$ and $n \geq \ell_1 + \ell_2$, it is not possible that $w^{(k,w_k+m)}$ is $\ell_2$-periodic.
\end{lem}
\begin{proof}
    We proceed by contradiction, and assume that $w^{(k,w_k+m)}$ is $\ell_2$-periodic.  We have two cases.  First, if $k > \ell_1$, then by Lemma \ref{MPL2}, there exists some $m' \in \A_d$ such that $w^{(k,w_k+m')}$ has period of length $\hbox{gcd}(\ell_1, \ell_2)$.  Then we have the following chain of equalities.
    $$\begin{array}{rcll}
        w_k & = & w_{k-\ell_1} & \hbox{since $w$ has a period of length $\ell_1$} \\
        & = & w_{k-\ell_2} & \hbox{since $w^{(k,w_k+m')}$ has a period of length $\hbox{gcd}(\ell_1, \ell_2)$} \\
        & = & w_k^{(k,w_k+m)} & \hbox{since $w^{(k,w_k+m)}$ has a period of length $\ell_2$}
    \end{array}$$
    Hence this is a contradiction.  For our second case, when $k \leq n- \ell_1$, we note the following. $$w_k = w_{k+\ell_1} = w_{k+\ell_2} = w_k^{(k,w_k+m)}$$  This is also a contradiction.  Therefore we must have that $w^{(k,w_k+m)}$ is not $\ell_2$-periodic.
\end{proof}

The following lemma leverages the previous one for an additional method of constructing non-periodic strings.

\begin{lem}\label{L2.16}
    Let $w \in \A_d^n$ such that $w$ has period length $\ell_1$ for some $\ell_1 < \frac{n}{2}$.  Suppose that $m \in [d-1]$,  $i,j,k \in [n]$ with $i \leq k \leq j$, and $\ell_2 \leq \ell_1$ with $j-i+1 \geq \ell_1 + \ell_2$.  If either $k \geq i+ \ell_1$ or $k \leq j-\ell_1$ then we must have that $w^{(k,w_k+m)}(i:j)$ does not have period $\ell_2$.
\end{lem}
\begin{proof}
    Define $w' = w(i:j)$ and hence $w'^{(k-i,w_{k-i}+m)}=w^{(k,w_k+m)}(i:j)$, and then apply Lemma \ref{MPL3} to compare the two strings.
\end{proof}

Lastly, we provide a lemma for creating more periodic strings from existing ones.

\begin{lem}\label{L8}
    Let $n=2t$ and let $w \in \A_d^n$.  If $w$ has period length $t$, and if for some $\ell<t$ and $m \in \A_d$ we find that $w' = w^{(t,w_t+m)}(t+1-\ell : n-1)$ is $\ell$-periodic, then we must have that $\ell$ divides $t$ and $w' \oplus (w_n+m)$ has period $\ell$.
\end{lem}
\begin{proof}
    First, we note that $W = \left(w^{(t,w_t+m)} \right)^{(n,w_n+m)}$ clearly has period length $t$, and so $W(t+1-\ell:n-1)$ has almost period length $t$.  Additionally, since $W(t+1-\ell:n-1)=w'$, we know that $W(t+1-\ell:n-1)$ is $\ell$-periodic.  Hence by Lemma \ref{MPL1}, $W(t+1-\ell:n-1)$ has period of length $p=\hbox{gcd}(t, \ell)$.  However since $w'$ is given to be $\ell$-periodic, the minimum period length is $\ell$ and so we must have that $p=\ell$ and thus $\ell$ divides $t$.

    To show that $w' \oplus (w_n+m)$ has period $\ell$, we note that $w'$ has period $\ell$, and that $w'_\ell=w_t^{(t,w_t+m)}=w_t+m$.  Having period $\ell$ implies that $w'_{k} = w_t+m$ for all $k$ that is divisible by $\ell$.  Since $w_n+m$ is the $(t+\ell)$th letter in $w' \oplus (w_n+m)$, and this is divisible by $\ell$, we need that $w_n+m=w_t+m$ in order for $w' \oplus (w_n+m)$ to have period $\ell$.  But this is given to be true since $w$ has period $t$.
\end{proof}

\section{Identifying Codes}\label{IDSet}

We begin by building up to the definition of an identifying code.  This requires careful definitions of directed distance and $t$-balls.  The graphs that we focus on in this paper are directed de Bruijn graphs.
  
    The directed de~Bruijn graph, denoted $\vec{\mathcal{B}}(d,n)$, has vertex set $\A_d^n$.  An edge from vertex $v=v_1v_2 \ldots v_n$ to vertex $w=w_1w_2 \ldots w_n$ exists if and only if the suffix of $v$ equals the prefix of $w$, or more formally $v_2v_3 \ldots v_n = w_1 w_2 \ldots w_{n-1}$.  In these graphs, we consider only \textit{directed} distance in which the distance from vertex $v$ to $w$ is denoted as $d(v,w)$ and is defined as the length of the shortest directed path from $v$ to $w$,

\begin{defn}
     The \textbf{in-ball of radius $t$} centered at vertex $v$ is the set {\boldmath $B_t^-(v)$} $= \{u \in V(G) \mid d(u,v) \leq t\}.$ 
     
     A special variation on this definition is when we consider the in-ball of radius 1, or the \textbf{in-neighborhood} of a vertex $v$.  The closed in-neighborhood is given by $N^-[v]= B_1^-(v)$, and the \textbf{open in-neighborhood} of $v$ is given by $N^-(v)= B_1^-(v) \setminus \{v\}$.
\end{defn}

The two following lemmas are useful in working with distances in $\B(d,n)$ and their proofs are self-evident.

\begin{lem}\label{lem:distance} \rm In $\B(d,n)$ there is a directed path of length $t\leq n$ from $v$ to $w$ if and only if $v(t+1:n) = w(1:n-t)$.  That is, if and only if the rightmost $n-t$ letters of $v$ are the same as the leftmost $n-t$ letters of $w$. \end{lem}

\begin{lem}\label{lem:samenbrhd}  \rm In $\B(d,n)$ if vertices $v\ne w$ have the same prefix, then for all $u\not \in \{v,w\}, d(u,v) = d(u,w)$.  In particular, $B_t^-(v)\setminus\{v\} = B_t^-(w)\setminus \{w\}$ for all $t\leq n$.\end{lem}

\begin{defn}\label{def:idcode}
Given a subset $S\subset V(G)$, the \textbf{$t$-identifying set with respect to $S$} for vertex $v$ is given by $\textsf{ID}_t^S(v) = B_t^-(v) \cap S$. 
 
 A \textbf{$t$-identifying code} is a set $S \subseteq V(G)$ such that each vertex has a unique, non-empty $t$-identifying set.  That is, for every $v\in V(G)$, $\textsf{ID}_t^S(v) \ne \emptyset$, and for all pairs $v,w \in V(G)$ we have $\textsf{ID}_t^S(v) \neq \textsf{ID}_t^S(w)$.  The variable $t$ is referred to as the \textbf{radius} of the identifying code.  
 
 A vertex $u \in \textsf{ID}_t^S(v) \triangle \textsf{ID}_t^S(w)$ is said to \textbf{separate} $v$ and $w$.  See Figure \ref{exIdCode} for an identifying code in the graph $\vec{\mathcal{B}}(2,3)$.
\end{defn}

\begin{figure}
\begin{center}

\begin{tikzpicture}[->,>=stealth',shorten >=1pt,auto,node distance=2cm,
  thick,main node/.style={circle,draw,font=\sffamily\bfseries,scale=0.75},new node/.style={circle,fill=black,text=white,draw,font=\sffamily\bfseries,scale=0.75}]

  \node[main node] (0) {000};
  \node[new node]  (1) [above right of=0] {001};
  \node[main node] (2) [below right of=1] {010};
  \node[new node]  (4) [below right of=0] {100};
  \node[main node] (5) [right of=2]       {101};
  \node[new node]  (6) [below right of=5] {110};
  \node[new node]  (3) [above right of=5] {011};
  \node[main node] (7) [below right of=3] {111};

  \path[every node/.style={font=\sffamily\small}]
    (0) edge node [left]      {} (1)
        edge [loop left] node {} (0)
    (1) edge node [left]      {} (3)
        edge node [right]     {} (2)
    (2) edge [bend right] node{} (5)
        edge node [right]     {} (4)
    (3) edge node [right]     {} (6)
        edge node [right]     {} (7)
    (4) edge node [left]      {} (0)
        edge node [right]     {} (1)
    (5) edge [bend right] node{} (2)
        edge node [right]     {} (3)
    (6) edge node [right]     {} (5)
        edge node [right]     {} (4)
    (7) edge [loop right] node{} (7)
        edge node [right]     {} (6);

\end{tikzpicture}

\end{center}
\caption{A 1-identifying code in the graph $\vec{\mathcal{B}}(2,3)$ (black vertices).  A 2-identifying code in this graph requires all vertices but one of $000,111$, and there are no $t$-identifying codes for $t \geq 3$.} \label{exIdCode}
\end{figure}
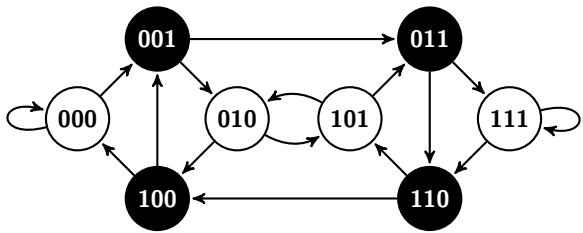

In the above definitions, if $t$ is omitted from the notation (i.e. identifying code instead of $t$-identifying code), then it is assumed that $t=1$.\medskip

Note that not every graph has a $t$-identifying code for each $t$.  In particular if the graph has two vertices with equal in-balls of radius $t$, then the graph has no $t$-identifying code.  The topic of such \lq $t$-twins' and the resulting non-existence of $t$-identifying codes is covered in Theorem \ref{MPT12}.

\begin{thm}
\label{DirIC} If $\B(d,n)$ is a $t$-identifiable graph, then the  size of any $t$-identifying code is at least $d^{n-1}(d-1)$.
\end{thm}

\begin{proof}  Choose $t\leq n$ and $a\ne b$ in $\A$.  Suppose that for some $u\in \A^{n-1}$, neither $v=u\oplus a$ nor $w=u\oplus b$ is in an arbitrary vertex subset $S$.  Since $v$ and $w$ share a prefix, by Lemma \ref{lem:samenbrhd}, $B_t^-(v)\setminus\{v\} = B_t^-(w)\setminus \{w\}$.  Since neither $v$ nor $w$ is in $S$, $\id(v) = \id(y)$. Thus $S$ is not a $t$-identifying code.  Thus for each possible $u\in \A^{n-1}$, a $t$-identifying code must contain, at least, all but one of $u\oplus a$ for $a\in \A$.  Thus a $t$-identifying code for $\B(d,n)$ must have size at least $d^{n-1}(d-1)$.\end{proof}

Note that the result above is independent of the radius $t$.  An interesting consequence of this is the fact that increasing the radius of our identifying code does not produce any decrease in the cardinality of a minimum identifying code.  For example, consider the potential application of identifying codes in sensor networks.  One might think that by increasing the sensing power (which corresponds to the radius of the identifying code) we would be able to place fewer sensors and thus incur a savings overall.  However, Theorem \ref{DirIC}  implies that providing more powerful  (and usually more expensive) sensors does not allow us to place fewer sensors.  Thus we should use sensors that have sensing distance equivalent to radius one.  In fact, in the case of 2-identifying codes in $\vec{\mathcal{B}}(2,3)$, we actually require an extra vertex for a minimum size of seven!

The remainder of this section is organized as follows.  We first provide a construction of a minimum $t$-identifying code for $\vec{\mathcal{B}}(d,n)$ with $t \geq 2$, and $n \geq 2t$ in Theorems \ref{MPMain} and \ref{thm:2ident}.  Following the proof of this result, we highlight some variations that provide identifying codes for several other instances.  Finally, we highlight an alternative construction for 1-identifying codes for all $\vec{\mathcal{B}}(d,n)$ when $d \neq 2$ and $n$ is odd.

\begin{thm}\label{MPMain}
    Suppose that $n \geq 5, d \geq 2, t \geq 2$, and $n \geq 2t$.  Then the following set $S$ is a minimum $t$-identifying code of size $d^{n-1}(d-1)$ in $\vec{\mathcal{B}}(d,n)$.
\begin{eqnarray*}
    S & = & \left\{w \in \A_d^n \mid \hbox{for some } m \hbox{ and } \ell \leq t, w^{(t,w_t+m)}(t+1-\ell:n-1) \hbox{ is} \right. \\
      &   & \hspace{10mm}  \left. \hbox{$\ell$-periodic,}\hbox{ but } w^{(t,w_t+m)}(t+1-\ell:n-1)\oplus (w_n+m) \hbox{ is not.}\right\} \\
      &   & \hspace{50mm}  \cup \\
      &   & \left\{ w \in \A_d^n \mid w_t \neq w_n \hbox{ and } w^{(t,w_t+m)}(t+1-\ell:n-1) \right. \\
      &   & \hspace{17mm}  \left. \hbox{ is not $\ell$-periodic for any $m$ and } \ell \leq t\right\}
\end{eqnarray*}
\end{thm}
\begin{proof}
    Consider an arbitrary $w \in \A_d^n$.  We note that by Lemma \ref{L2.16} that $w^{(t,i)}(t+1-\ell : n-1)$ is $\ell$-periodic for at most one $\ell \leq t \leq \frac{n}{2}$ and for at most one $i \in \mathcal{A}_d$.

In order to determine the cardinality of the set $S$, we will determine which strings $w$ are \textit{not} included in the set $S$.  We have two cases:  either $w^{(t,w_t+m)}(t+1-\ell:n-1)$ is $\ell$-periodic for some $m$ and $\ell$, or it is not for any $m$ and $\ell$.

If $w^{(t,w_t+m)}(t+1-\ell:n-1)$ is $\ell$-periodic for some $m$ and $\ell$, then there exists a unique $b$ such that $w^{(t,w_t+m)}(t+1-\ell:n-1) \oplus b$ is also $\ell$-periodic.  In this case, $w \not \in S$ if and only if $w_n=b-m$.  Hence for each $w(1:n-1)$ in this case, there is exactly one value for $w_n$ that creates a string $w \not \in S$.

On the other hand, if $w^{(t,w_t+m)}(t+1-\ell:n-1)$ is not $\ell$-periodic for any $m$ and $\ell$, then $w \not \in S$ if and only if $w_t = w_n$.  Hence for each $w(1:n-1)$ in this case, there is exactly one value for $w_n$ that creates a string $w \not \in S$.

Since each string $w(1:n-1)$ must fall into exactly one of the two cases, and each case there is only one choice for $w_n$ to create a string outside of the set $S$, we must have that $d^{n-1}$ strings are not in $S$, leaving the cardinality of $S$ at $d^n-d^{n-1}$, or $d^{n-1}(d-1)$.

Additionally, it is important to note that there is only one choice for $r$ such that $w^{(n,r)} \not \in S$, and only one choice for $s$ such that $w^{(t,s)} \not \in S$. These two facts follow when we observe that if $w^{(t,a)}(t+1-\ell : n-1) \oplus b$ is $\ell$-periodic, then we must have that $w \not \in S$ if and only if $w_t = w_n+a-b$.

    Now that we have established the cardinality of $S$, we must show that no two vertices have the same identifying sets.  Let $v,w \in \vec{\mathcal{B}}(d,n)$, and let $k$ be the smallest index such that $v_k \neq w_k$.

    \begin{description}
        \item[$k=1$:]  Without loss of generality, we may assume that $v_1=0$ and $w_1=1$.  Observe that we have the following strings contained in the in-balls of radius $t$ for each.
        $$\begin{array}{rcccl}
            v' & = & 0^{t-1} \oplus 0 \oplus v(1:n-t) & \in & B_t^-(v), \\
            v''& = & 0^{t-1} \oplus 1 \oplus v(1:n-t) & \in & B_t^-(v), \\
            w' & = & 1^{t-1} \oplus 1 \oplus w(1:n-t) & \in & B_t^-(w), \hbox{ and} \\
            w''& = & 1^{t-1} \oplus 0 \oplus w(1:n-t) & \in & B_t^-(w).
        \end{array}$$

        Note that $v' \not \in B_t^-(w)$ and $w' \not \in B_t^-(v)$, since $B_t^-(v)$ does not contain any vertices beginning with $1^{t+1}$ and $B_t^-(w)$ does not contain any vertices beginning with $0^{t+1}$.  Next, we notice that at least one of $v', v''$ is in $S$.  To see this, we note that $v''=v'^{(t,1)}$.  By the same point, we must have that at least one of $w', w''$ is a member of $S$.

        Next, we note that if at least one of $v', w'$ is a member of $S$, we can use that string to separate $v$ and $w$.  Otherwise, if either $v'' \not \in B_t^-(w) \cap S$ or $w'' \not \in B_t^-(v) \cap S$, we can separate $v$ and $w$ with the given string.  As a last resort, we consider the case in which we have $v'' \in B_t^-(w) \cap S$ and $w'' \in B_t^-(v) \cap S$.  Then we must have:
        \begin{eqnarray*}
            v'' & = & 0^{t-1} \oplus w(1:n-t+1), \\
            w'' & = & 1^{t-1} \oplus v(1:n-t+1),
        \end{eqnarray*}
        as $v_t''$ is the only 1 in $v''(1:t+1)$ and $w_t''$ is the only 0 in $w''(1:t+1)$.  From this, we get the following string equalities:
        \begin{eqnarray*}
            v(1:n-t) & = & w(2:n-t+1), \\
            w(1:n-t) & = & v(2:n-t+1).
        \end{eqnarray*}
        From these string equalities, we see that $v_i=w_{i+1}=v_{i+2}$ and $w_i=v_{i+1}=w_{i+2}$ for all $i=1, 2, ... , n-t-1$, and so $v(1:n-t)$ and $w(1:n-t)$ are both 2-periodic (since $v_1 \neq w_1$ we cannot have 1-periodic).  Hence, $v''(t-1:n) = 0 \oplus 1 \oplus v(1:n-t)$ is also 2-periodic (recall that $v_1=0$ and $w_1=1$), so $v'' \not \in S$.  Hence we must have $v' \in S$, and so we may use $v'$ to separate $v$ and $w$.

        \item[$2 \leq k \leq n-t$:]  We know that there must exist some $s$ such that $v_1 \ldots v_{s-1} = w_1 \ldots w_{s-1}$, and these substrings are constant, that is $v_1=\ldots = v_{s-1}=w_1=\ldots=w_{s-1}$.  Without loss of generality we may assume that $v_1 \ldots v_{s-1} = 0^{s-1} = w_1 \ldots w_{s-1}$ and $v_s = 1$, and so $2 \leq s \leq k$.  Define the following strings.
            \begin{eqnarray*}
                w' & = & 1^t \oplus w(1:n-t) \hbox{ and} \\
                w'' & = & 1^{t-1} \oplus 0 \oplus w (1:n-t)
            \end{eqnarray*}
            As we saw in the previous case, we have that $\{w', w''\} \subseteq B_t^-(w)$ and $\{w', w'' \} \cap S \neq \emptyset$.  Now consider an arbitrary vertex $u \in B_t^-(v)$.  Since $v_{s-1}v_s=01$ and $2 \leq s \leq k \leq n-t$, we know that $u(i-1:i)=01$ for some $i \in [s, s+t]$.

            Additionally, we consider $w'$ and $i \in [2, s+t-1]$.  For $i \leq t$, we know that $w'(i-1:i)=11$, and for $i=t+1$, we have that $w'(i-1:i)=10$, and finally for $t+2 \leq i \leq s+t-1$ we have $w'(i-1:i)=00$.  Similarly, for $i \in [2,s+t-1]$, we must have $w''(i-1:i) \in \{00,10,11\}$.  This implies that $w'(1:s+t-1)$ and $w''(1:s+t-1)$ do not contain the substring $01$, and so if $w'$ or $w''$ is a member of $B_t^-(v)$, we must have either $d(w',v)=t$ or $d(w'',v)=t$, respectively.  Hence we must have $w'_{t+i}=v_i$ or $w''_{t+i}=v_i$ for $i \in [1, n-t]$, and therefore that $v_k=w'_{t+k}=w''_{t+k}=w_k$, which is a contradiction.  Thus neither $w'$ nor $w''$ can be a member of $B_t^-(v)$, and hence both strings separate $v$ and $w$.

            \item[$k > n-t$:]  Since we must have $v_1=w_1$, we may assume without loss of generality that $v_1=w_1 \neq 0$.  Define the following strings.
            \begin{eqnarray*}
                v' & = & 0^{n-k} \oplus v(1:k) \in B_t^-(v) \hbox{ and} \\
                w' & = & 0^{n-k} \oplus w(1:k) \in B_t^-(w)
            \end{eqnarray*}
            We have $v'(1:n-1)=w'(1:n-1)$ and $v'_n = v_k \neq w_k = w'_n$.  Note that this implies that for $a=v_{t-n+k}$ we have $v'=\left(v'^{(t,a)}\right)^{(n,v_k)}$ and $w'=\left(v'^{(t,a)}\right)^{(n,w_k)}$ with $v_k \neq w_k$.  By our argument at the very beginning of the proof, at most one of these can lie outside of $S$, and so we must have $\{v',w'\} \cap S \neq \emptyset$.

            We now have two cases.  First, if both $v' \not \in B_t^-(w)$ and $w' \not \in B_t^-(v)$, then any string from $\{v',w'\} \cap S$ separates $v$ and $w$, and we are done.  Otherwise, assume without loss of generality that $v' \in B_t^-(w)$.  Then we must have $v'=x \oplus w(1:p)$ for at least one $p \in [n-t,n]$ and some substring $x$.  Take $p$ to be the largest such $p$ possible.  Since $w_1 \neq 0 = v'_i$ for all $i \in [1,n-k]$, we must have $p \leq k$.  Additionally, if $p=k$ then we have $w(1:k)=w(1:p)=v(1:k)$, which is a contradiction since $w_k \neq v_k$.  This implies that we must have $p<k$.  Hence we must have the following string of equalities:
            \begin{eqnarray*}
                0^{n-k} \oplus v(1:k) & = & v' \\
                & = & x \oplus w(1:p) \\
                & = & x \oplus v(1:p).
            \end{eqnarray*}
            Thus we have $v(k-p+1:k)=v(1:p)$, or $v_i=v_{k-p+i}$ for $i=1,2,\ldots , p$.  Additionally we note the following equalities hold:
            \begin{eqnarray*}
                2(k-p) & = & 2k-2p \\
                & \leq & 2k-2(n-t) \\
                & \leq & k+2t-n \\
                & \leq & k.
            \end{eqnarray*}
            Hence since $v_i=v_{k-p+i}$ for $i \in [1, p]$ and $2(k-p) \leq k$, we know that $v(1:k)$ has period $(k-p)$.  In fact, since we chose $p$ to be maximum, $v(1:k)$ is $(k-p)$-periodic.

            Next, we show that $v'(t+1-(k-p):n)$ is also $(k-p)$-periodic.  First, we note the following inequalities:
            \begin{eqnarray*}
                n-(t+1-(k-p))+1 & = & n-t+k-p \\
                & \geq & 2(k-p).
            \end{eqnarray*}
            The last line comes from the following fact: \begin{equation} \label{Eq:1}
            n-t \geq t \geq n-p \geq k-p. 
\end{equation} Hence our string length is at least $2(k-p)$, and from our previous paragraphs, so long as $v'(t+1-(k-p):n)$ is contained in $v'(n-k+1:n)=v(1:k)$, we know that it must have period $(k-p)$.  For this we note that $$t+1-(k-p) \geq (n-p)+1-(k-p) = n-k+1,$$ and so $v'(t+1-(k-p):n)$ indeed has period $(k-p)$, and thus $v' \not \in S$, except if $(n-1)-(t+1-(k-p))+1 < 2(k-p)$.  In this case, all inequalities must be equalities in Equation (\ref{Eq:1}).  From this it follows that $k-p=t$, $k=n=2t$, and $p=t$.

            If $v'^{(t,v'_t+m)}(t+1-\ell:n-1)$ is $\ell$-periodic for some $\ell < k-p=t$ and $m$, then by Lemma \ref{L8} we must have that $v'^{(t,v'_t+m)}(t+1-\ell:n-1) \oplus (v'_n+m)$ is also $\ell$-periodic, and hence $v' \not \in S$.  On the other hand, if $v'^{(t,v'_t+m)}(t+1-\ell:n-1)$ is not $\ell$-periodic for any $\ell < t$ and $m$, then we note that $v'_n=v'_{2t}=v'_t$, so again $v' \not \in S$.  Therefore in all cases we have $v' \not \in S$, so we must have $w' \in S$.

            All that remains is to show that $w' \not \in B_t^-(v)$.  We note that if $w' \in B_t^-(v)$, then $w(1:k)$ is $\ell$-periodic for some $\ell \leq k+t-n$ (using the same argument as we used to show that $v(1:k)$ was $(k-p)$-periodic).  Since $v(1:k)=\left( w(1:k) \right)^{(k,w_k+m)}$ for some $m$, by Lemma \ref{MPL3} it is not possible that both $v(1:k)$ and $w(1:k)$ are periodic.  Hence we must have $w' \not \in B_t^-(v)$, and so we may use $w'$ to separate $v$ and $w$.

    \end{description}
\end{proof}

When we combine the previous theorem with the following result, we have a complete set of constructions for minimum $t$-identifying codes in $\vec{\mathcal{B}}(d,n)$ with $t \geq 2$ and $n \geq 2t$.

\begin{thm}\label{thm:2ident} Let $n>3$ and $S = \A_d^n \setminus \{\left(w^{(2,a)}\right)^{(n,a)} \mid w \in \A_d^n \text{ and } a \in \A_d\}$. If $n$ is even, then $S$ is a $2$-identifying code for $\B(d,n)$. If $n$ is odd, then $S'= (S \cup \{(ab)^{\frac{n-1}{2}}b \mid \ a\ne b \in  \A_2\}) \setminus \{(ab)^{\frac{n-1}{2}}a \mid \ a\ne b \in  \A_2\}$ is a $2$-identifying code for $\B(d,n)$.  In both these cases, the 2-identifying code is of minimum size $d^{n-1}(d-1)$.\end{thm}

\begin{proof}
    Consider an arbitrary string $w \in \A_d^n$. We'll consider the cardinality of $\textsf{ID}_t^S(w)$ in four cases  based on the relations between $w_1, w_{n-1}$ and between $w_2, w_n$.  First let $C= \{a \oplus w(1:n-2) \mid a\in \A\setminus\{w_{n-2}\} \}$.\medskip

{\bf Case 1. } If $w_2 = w_n$ and $w_1 = w_{n-1}$, then $\textsf{ID}_t^S(w)=\A\oplus C$.  Thus $|\textsf{ID}_t^S(w)|=d^2-d$.\medskip 

{\bf Case 2.} If $w_2 \neq w_n$ and $w_1 = w_{n-1}$, then $\textsf{ID}_t^S(w) = (\A\oplus C) \cup \{w\}$. 
If $w\in \A\oplus C$ then $w(2:n) = a \oplus w(1:n-2)$ for some $a\in \A\setminus\{w_{n-2}\}$.  This implies that we have $w_1 = w_3 = w_5 = \cdots$, and also that $w_2 = w_4 = w_6 = \cdots$.  Since this case requires that $w_1 = w_{n-1}$, we must have that either $n$ is even or that $w_1 = w_2 = w_3 = \cdots = w_n$.  In either case, this contradicts our assumption that $w_2 \neq w_n$.  Thus $w\not\in \A\oplus C$, and we conclude that $|\textsf{ID}_t^S(w)|=d^2-d+1$.\medskip

{\bf Case 3.} If $w_2 = w_n$ and $w_1 \neq w_{n-1}$, then $\textsf{ID}_t^S(w)= \A \oplus \{C\cup\{w(1:n-1)\}\}$.
If $w(1:n-1) \in C$, then $aw_1w_2 \cdots w_{n-2} = w_1w_2 \cdots w_{n-1}$.  This implies that we have $w_1 = w_2 = w_3 = \cdots = w_{n-2} = w_{n-1}$.  This contradicts our assumption that $w_1 \neq w_{n-1}$.  Thus $w(1:n-1) \not \in C$, and we conclude that $|\textsf{ID}_t^S(w)| = d^2$.\medskip

{\bf Case 4.} If $w_2 \neq w_n$ and $w_1 \neq w_{n-1}$, then $\textsf{ID}_t^S(w)=  (\A \oplus \{C\cup\{w(1:n-1)\}\}) \cup \{w\}$.
As in Case 3, since $w_1\neq w_{n-1}, \A \oplus \{C\cup\{w(1:n-1)\}\}$ contains $d^2$ distinct elements.  Let us consider whether $w\in \A \oplus \{C\cup\{w(1:n-1)\}\}$. If not, then $|\textsf{ID}_t^S(w)| = d^2+1$. There are two cases to consider.\medskip


\quad {\bf a.}  If $w\in \A\oplus C$, then $w(2:n) = w(1:n-1)$.  This implies that we have the following chain of equalities:  $w_1 = w_2 = w_3 = \cdots = w_{n-1} = w_n$.  This contradicts the assumptions that $w_2 \neq w_n$ and $w_1 \neq w_{n-1}$.  Thus, this case does not occur.\medskip

\quad {\bf b.}  If $w\in \A\oplus \{w(1:n-1)\}$, then $w(2:n) = a\oplus w(1:n-2)$ for some $a\in \A$, and so $w_2w_3 \cdots w_n = aw_1w_2 \cdots w_{n-2}$.  This implies that $w_1 = w_3 = w_5 = \cdots$, and also that $w_2 = w_4 = w_6 = \cdots$.  If $n$ is even, this contradicts our assumptions that $w_2 \neq w_n$ and $w_1 \neq w_{n-1}$.  Thus for even $n$, this case does not occur.  

For $n$ odd, this case only occurs if $w =  (ab)^{\frac{n-1}{2}}a$ for some $b \in \A$.  In this case, we must modify $S$ to get an identifying code.  Note that $$B_2^- ((ab)^{\frac{n-1}{2}}a) \cup \{(ab)^{\frac{n-1}{2}}b\} = B_2^- ((ab)^{\frac{n-1}{2}}b).$$  Since our set $S$ contains $(ab)^{\frac{n-1}{2}}a$ but not $(ab)^{\frac{n-1}{2}}b$, we have $$\textsf{ID}_t^S\left((ab)^{\frac{n-1}{2}}a\right) = \textsf{ID}_t^S\left((ab)^{\frac{n-1}{2}}b\right).$$  Thus by adding $(ab)^{\frac{n-1}{2}}b$, we are able to create distinct identifying sets with respect to $S'=S \cup \{(ab)^{\frac{n-1}{2}}b\}.$  However, we note that we now have $$(ab)^{\frac{n-1}{2}}b, (ab)^{\frac{n-1}{2}}a \in S',$$ and also that $$B_2^+ \left((ab)^{\frac{n-1}{2}}a\right) \cup \left\{b(ba)^{\frac{n-1}{2}}\right\} = B_2^+ \left(b(ba)^{\frac{n-1}{2}}\right).$$  This implies that the inclusion of both $(ab)^{\frac{n-1}{2}}b$ and $(ab)^{\frac{n-1}{2}}a$ in our identifying code is only required if needed to identify vertex $(ab)^{\frac{n-1}{2}}a$ from vertex $b(ba)^{\frac{n-1}{2}}$, as otherwise they identify the same sets of vertices.  For $n>3$ we can identify $(ab)^{\frac{n-1}{2}}a$ differently from $b(ba)^{\frac{n-1}{2}}$ without using $b(ba)^{\frac{n-1}{2}}$ (they have disjoint in-balls of radius 2 in this case), we need only include $(ab)^{\frac{n-1}{2}}b$ and not $(ab)^{\frac{n-1}{2}}a$ in our identifying code.  Thus $S'$ is a 2-identifying code for $n>3$.\medskip

Reviewing all four cases presented, we can see that an identifying set completely determines its corresponding string $w$.  Hence we have provided a 2-identifying code $S$ when $n$ is even, and a separate 2-identifying code $S'$ when $n$ is odd.\end{proof}

Finally, we provide additional constructions of identifying codes.  Theorems \ref{MPT10} and \ref{MPT11} have proofs very similar to that of Theorem \ref{MPMain}, so we omit them here.

\begin{thm}\label{MPT10}
        Assume that $d \geq 2$, and $n \geq 3$.  Then the following subset $S$ is an optimal 1-identifying code of size $d^{n-1}(d-1)$ in $\vec{\mathcal{B}}(d,n)$.
    \begin{align*}
    S  = & \left\{w \in \A_d^n \mid \hbox{for some } m \hbox{ and } \ell \in \{1,2\}, w^{(1,w_1+m)}(1:n-1) \hbox{ is $\ell$-periodic} \right. \\
        & \hspace{10mm}  \left. \hbox{or almost $\ell$-periodic, but } w^{(1,w_1+m)}(1:n-1)\oplus (w_n+m) \hbox{ is not.}\right\} \\
        & \hspace{50mm} \cup \\
        & \left\{ w \in \A_d^n \mid w_1 \neq w_n \hbox{ and } w^{(1,w_1+m)}(1:n-1) \right. \\
        & \hspace{17mm}  \left. \hbox{ is not $\ell$-periodic for any $m$ and } \ell \in \{1,2\} \right\}
\end{align*}
\end{thm}

\begin{thm}\label{MPT11}
    Assume that $d\geq 2$. Then the following subset $S$ is a $t$-identifying code of size $d^{n-1}(d-1)+d^{t}$ in the directed de Bruijn graph $\vec{\mathcal{B}}(d,n)$, if $n=2t-1\geq 5$.
\begin{align*}
S =&\{w\in\A_d^n \,|\, \text{for some $m$ and $l< t-1$: } w^{(t,w_t+m)}(t+1-l:n-1) \text{ } \\
		& \hspace{5mm}	\text{is $l$-periodic, but $w^{(t,w_t+m)}(t+1-l:n-1)\oplus (w_n+m)$ is not.}\} \\
		&\hspace{50mm} \cup \\
       & \left\{w\in\A_d^n \,|\, w_t\neq w_n \text{ and $w^{(t,w_t+m)}(t+1-l:n-1)$ is not $l$-periodic}  \right. \\
       & \hspace{17mm} \left. \hbox{for any $m$ and $l< t-1$.}\right\} \\
		&\hspace{50mm} \cup \\
       &\{w\in\A_d^n \,|\, w^{(t,w_t+m)}(1:n-1) \text{ is almost $t$-periodic, for some $m$.} \}
\end{align*}
\end{thm}

We note that the construction in Theorem \ref{MPT11} is not optimal.  To find an optimal $t$-identifying code when $n=2t-1$ is an open problem to be considered in the future.  For the cases when $n < 2t-1$, we have the following theorem.

\begin{thm}\label{MPT12}
    There is no $t$-identifying code in the directed de Bruijn graph $\vec{\mathcal{B}}(d,n)$ when $n\leq 2t-2$.
\end{thm}
\begin{proof}
Let $v=0^{n-t}\oplus 1 \oplus 0^{t-2} \oplus 1$ and $w = 0^{n-t}\oplus 1 \oplus 0^{t-2} \oplus 0$.  Since $B_t^-(v)$ and $B_t^-(w)$ contain all vertices that end with $0^{n-t}$ or $0^{n-t}\oplus 1 \oplus 0^k$ where $k=0,1,\ldots, n-t-1$, $v$ and $w$ are $t$-twins.  Thus $\overrightarrow{\mathcal{B}}(d,n)$ has no $t$-identifying code.
\end{proof}

As an additional treat for the reader, we provide a simple construction for 1-identifying codes in $\vec{\mathcal{B}}(d,n)$ whenever we have either $d>2$ or $n$ odd.

\begin{thm}\label{1ident} If $n$ is odd, or $n$ is even and $d>2$, then $$S = \A_d^n \setminus \{a \oplus \A_d^{n-2} \oplus a \mid a \in \A_d\}$$ is an identifying code  for $\B(d,n)$.  Further this identifying code has optimal size  $(d-1)d^{n-1}$.\end{thm}

\begin{proof} Define $S$ as in the statement of the theorem.     First, we will see that the identifying set for every vertex has size either $d$ or $d-1$.  Let $w=w_1w_2 \ldots w_n$, then $$N^-(w) \cap S = \{\A \oplus w_1w_2 \ldots w_{n-1} \} \setminus \{w_{n-1}w_1w_2 \ldots w_{n-1}\}.$$

If $w_1 = w_n$, then $\textsf{ID}_t^S(w) = N^-(w)\cap S$ has size $d-1$. Whereas, if $w_1 \neq w_n$,  then $\textsf{ID}_t^S(w) = \{w\} \cup N^-(w) \cap S$ has size $d$.\medskip

From this it is clear that every vertex has a non-empty identifying set. However we must also show that every identifying set is unique.  Suppose there are two distinct vertices $v,w \in V(\vec{\mathcal{B}}(d,n))$ such that $\\textsf{ID}_t^S(v) = \textsf{ID}_t^S(w)$.  Call their identical identifying set $T$.  We look at the two cases, $|T|=d$ and $|T|=d-1$, separately below.\medskip

Suppose that $|T|=d$.  Then $\{v,w\} \subseteq T$ by our assumption on $T$ and our earlier reasoning.  Since $v\ne w$, this means that $\B(d,n)$ contains both directed arcs $v \rightarrow w$ and $w \rightarrow v$.  This allows us to conclude that $\{v,w\} = \{(ab)^k, (ba)^k\}$ for some distinct $a,b \in \A$ with $k=n/2$.  In particular we must have $n$ even.  Below are the precise identifying sets for $v$ and $w$.

        \begin{eqnarray*}
            \textsf{ID}_t^S((ab)^k) & = &  \{(ab)^k, (ba)^k\} \cup \{c(ab)^{k-1}a \mid c \in \A \setminus \{a,b\} \} \\
            \textsf{ID}_t^S((ba)^k) & = &  \{(ab)^k, (ba)^k\} \cup \{c(ba)^{k-1}b \mid c \in  \A \setminus \{a,b\} \}
        \end{eqnarray*}

        If $d>2$ these two identifying sets are in fact different, which is a contradiction.\medskip

Suppose that $|T|=d-1$.  Then neither $v$ nor $w$ is in $T$, which means neither is in $S$.  However since their identifying sets are identical, this means that they have identical first neighborhoods. By definition of first neighborhoods, this means that $v$ and $w$ have the same prefix but different final letters. By then definition of $S$, one of $v,w$ (if not both) is a member of $S$, which is a contradiction.
\end{proof}

\section{Dominating, Resolving, and Determining Sets}\label{OtherSet}

In this section we examine other types of vertex sets which identify or classify vertices up to some graph property.  The properties used to define these sets are adjacency, distance, and automorphisms.

\newcommand{\sym}{{\rm Sym}}
\newcommand{\ptstab}{{\rm PtStab}}
\newcommand{\Det}{{\rm Det}}
\newcommand{\Z}{{\mathbb Z}}

\subsection{Dominating Sets}\label{Dominating}

\begin{defn}  A (directed) \textbf{$t$-dominating set} is a subset $S \subseteq V(G)$ such that for all $v \in V(G)$ we have $B_t^-(v) \cap S \neq \emptyset$.  That is, $S$ is a (directed) $t$-dominating set if every vertex in $G$ is within (directed) distance $t$ of some vertex in $S$.  We denote the size of a minimum $t$-dominating set in a graph $G$ by $\gamma_t(G)$.
\end{defn}

Note that by definition every identifying code is also a dominating set, but not conversely.

\begin{thm}\label{domexact}
    \emph{\cite{DomConstr}}  For $d \geq 2$, $n \geq 1$, $\gamma_1(\vec{\mathcal{B}}(d,n)) = \left\lceil \frac{d^n}{d+1} \right\rceil$.
\end{thm}

In \cite{DomConstr} a construction of a minimum dominating set for $\B(d,n)$ is given.  Key to this construction is the fact that every integer $m$ corresponds to a string (base $d$) in $\Z_d^n$, that we call $X_m$.   The construction utilizes a special integer $m$ defined by:

$$m = \left\{
        \begin{array}{ll}
          d^{n-2}+d^{n-4}+ \cdots + d^{n-2k} + \cdots + d^2 + 1 \bmod d^n, & \hbox{if $n$ is even;} \\
          d^{n-2}+d^{n-4}+ \cdots + d^{n-2k} + \cdots + d^3 + d \bmod d^n, & \hbox{if $n$ is odd.}
        \end{array}
      \right.$$
Let $D = \{m, m+1 , \ldots , m+ \lceil \frac{d^n}{d+1} \rceil -1\}$.  Now let $S$ be the set of strings $\{X_i \mid i \in D\}$.  Then $S$ is a minimum size dominating set for $\vec{\mathcal{B}}(d,n)$.\medskip

Next we provide constructions for $t$-dominating sets.  While others have considered some variations of $t$-dominating sets (such as perfect dominating sets in \cite{PDS}), it does not appear that the general $t$-dominating sets have been considered in the directed de Bruijn graph.

\begin{thm}
    The set $S\cup \{0^n\}$ where
\[
	S=\{w\in\A_d^n \,|\, w_{k(t+1)}\neq 0 \text{ for some $k\in\Z_+$ and } w_{i}=0 \text{ for all $i<k(t+1)$} \}
\]
is a $t$-dominating set of size
$$1+d^{n-t-1}(d-1)\left(\frac{1-d^{-(t+1)\left\lfloor \frac{n}{t+1}\right\rfloor}}{1-d^{-(t+1)}}\right)$$
in $\vec{\mathcal{B}}(d,n)$.
\end{thm}
\begin{proof}
Let $w$ be a string in $\A_d^n$. Assume that there are $k$ zeros at the beginning of $w$, but not $k+1$ zeros, i.e. $w=0^n$ or $w=0^k\oplus a\oplus w(k+2:n)$ for some $a\neq 0$. Let $l\in[0,t]$ be an integer so that $k+l\equiv t \pmod{t+1}$, i.e. $k+l=t+m(t+1)=(m+1)(t+1)-1$ for some $m\geq 0$. Now
$$0^l\oplus w(1:n-l)=0^l\oplus 0^k\oplus a\oplus w(k+2:n-l) =0^{k+l}\oplus a\oplus w(k+2:n-l)$$
belongs to $S$ except if $k+l\geq n$. If $k+l\geq n$, then $0^{n-k} \oplus w(1:k)=0^n\in S\cup\{0^n\}$ dominates $w$. Therefore every vertex is dominated by $S\cup\{0^n\}$.

There are $d^{n-k(t+1)}\cdot(d-1)$ strings which begin with exactly $k(t+1)-1$ zeros. Moreover, every string in $S\setminus\{0^n\}$ begins with at most $n-1$ zeros. This implies that $k(t+1)<n$ or $1\leq k \leq \left\lfloor \frac{n}{t+1}\right\rfloor$. Finally, $0^n$ is added to the dominating set $S\cup\{0^n\}$. Therefore the size of $S\cup\{0^n\}$ is
\begin{align*}
	& 1+\sum_{k=1}^{\left\lfloor \frac{n}{t+1}\right\rfloor} d^{n-k(t+1)}\cdot (d-1) \\
	& = 1+d^n(d-1)\sum_{k=1}^{\left\lfloor \frac{n}{t+1}\right\rfloor} \left(d^{-t-1}\right)^k \\
	& = 1+d^n(d-1)\left(-1+\sum_{k=0}^{\left\lfloor \frac{n}{t+1}\right\rfloor} \left(d^{-t-1}\right)^k \right) \\
	&	= 1+d^n(d-1)\left(-1+\frac{1-\left(d^{-t-1}\right)^{\left\lfloor \frac{n}{t+1}\right\rfloor +1}}{1-d^{-(t+1)}}\right) \\
	&	= 1+d^n(d-1)\left(\frac{d^{-(t+1)}-\left(d^{-(t+1)}\right)^{\left\lfloor \frac{n}{t+1}\right\rfloor +1}}{1-d^{-(t+1)}}\right) \\
	&	= 1+d^{n-t-1}(d-1)\left(\frac{1-d^{-(t+1)\left\lfloor \frac{n}{t+1}\right\rfloor}}{1-d^{-(t+1)}}\right). 
\end{align*}
\end{proof}

Note that the previous result is optimal if $n=t \mod t+1$.  Additionally, this result gives us the following lower bound on the size of a $t$-dominating set in $\vec{\mathcal{B}}(d,n)$.

\begin{thm}
    Bounds on the size of a $t$-dominating set in $\vec{\mathcal{B}}(d,n)$ are given by:
\begin{align*}
	\gamma_t(\vec{\mathcal{B}}(d,n))\geq
	\begin{cases}
		1+d^{n-t-1}(d-1)\left(\frac{1-d^{-(t+1)\left\lfloor \frac{n}{t+1}\right\rfloor}}{1-d^{-(t+1)}}\right)	& \text{\emph{if }} n\equiv t \pmod{t+1} \\
		d^{n-t-1}(d-1)\left(\frac{1-d^{-(t+1)\left\lfloor \frac{n}{t+1}\right\rfloor}}{1-d^{-(t+1)}}\right)	& \text{\emph{otherwise.}}
	\end{cases}
\end{align*}
\end{thm}
\begin{proof}
    Suppose that the set $S'$ is a $t$-dominating set in $\vec{\mathcal{B}}(d,n)$.  Choose $a \in \A_d \setminus \{0\}$ and $i \in \Z$ so that $i \leq \frac{n}{t+1}$, and $w \in \A_d^{n-i(t+1)}$.  Let $x=w \oplus a \oplus 0^{i(t+1)-1}$.  We note that $B_t^-(x)$ contains the following elements. $$B_t^-(x) = \left\{y \mid y=w' \oplus w \oplus a \oplus 0^{i(t+1)-1-k} \hbox{ for some } k \in [0,t], w' \in \A_d^k\right\}$$  Note for all $v \neq x$ such that $v=v' \oplus b \oplus 0^{j(t+1)-1}$ with $b \in \A_d \setminus \{0\}$, $j \leq \frac{n}{t+1}$, and $v' \in \A_d^{n-j(t+1)}$, we must have that $B_t^-(x) \cap B_t^-(v) = \emptyset$.  Hence each of these types of strings must dominated by a different element of $S'$, and so we must have the following lower bound on $|S'|$.  Define $A= \left\{v \mid v= v'\oplus b \oplus 0^{j(t+1)-1} \hbox{ with } b \in \A_d\setminus {0}, j \leq \frac{n}{t+1}, v' \in \A_d^{n-j(t+1)} \right\}$.
    \begin{eqnarray*}
        |S'| & \geq & \left| A \right| \\
        & = & \sum_{j=1}^{\lfloor \frac{n}{t+1} \rfloor} d^{n-j(t+1)}\cdot (d-1) \\
        & = & d^{n-t-1}(d-1)\left(\frac{1-d^{-(t+1)\left\lfloor \frac{n}{t+1}\right\rfloor}}{1-d^{-(t+1)}}\right)
    \end{eqnarray*}

    Finally, we consider the string $0^n$ and note that $$B_t^-(0^n)=\{z \mid z=z' \oplus 0^{n-s} \hbox{ with } z' \in \A_d^s, s \leq t\}.$$  When we compare $B_t^-(0^n)$ with $B_t^-(x)$, we note that since $a \neq 0$ we must have that the closest element of $B_t^-(x)$ to $0^n$ is $x$ itself.  Next we note that the string closest to $0^n$ in the set $A$ will occur when $j=\lfloor \frac{n}{t+1} \rfloor$. This will give us the string with the most 0's packed at the right end.  Finally, if $n \equiv p \bmod {t+1}$, then this string looks like $v' \oplus b \oplus 0^{n-p-1}$ with $v' \in \A_d^p$ and $b \neq 0$.  If $p=t$, then we are still unable to reach $0^n$, and so we must have at least one additional string in $S'$ to cover $0^n$.
\end{proof}

\subsection{Resolving Sets}\label{Resolving}

\begin{defn}
     A \textbf{directed resolving set} is a set $S$ such that for all $u,v\in V(G)$ there exist $s\in S$ so that $\vec d(s,u)\ne \vec d(s,v)$.  The \textbf{directed metric dimension} is the minimum size of a directed resolving set.  An example of a directed resolving set in $\vec{\mathcal{B}}(2,3)$ is given in Figure \ref{exResSet}.
\end{defn}

Note that this definition is not quite the same as that given in \cite{DirResolving} (which requires that there exist $s\in S$ so that $d(u,s) \ne d(v, s)$).  Our definition corresponds better to the definitions of domination and of identifying codes for directed graphs  that are used in this paper, however results using the standard definition may be found in \cite{Feng}.

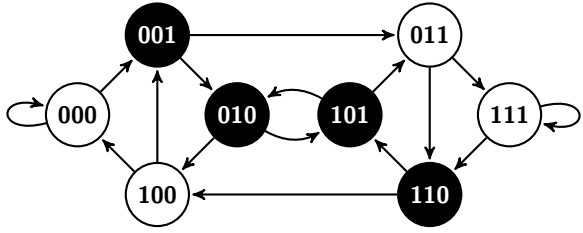
\begin{figure}
\begin{center}

\begin{tikzpicture}[->,>=stealth',shorten >=1pt,auto,node distance=2cm,
  thick,main node/.style={circle,draw,font=\sffamily\bfseries,scale=0.75},new node/.style={circle,fill=black,text=white,draw,font=\sffamily\bfseries,scale=0.75}]

  \node[main node] (0) {000};
  \node[new node]  (1) [above right of=0] {001};
  \node[new node] (2) [below right of=1] {010};
  \node[main node]  (4) [below right of=0] {100};
  \node[new node] (5) [right of=2]       {101};
  \node[new node]  (6) [below right of=5] {110};
  \node[main node]  (3) [above right of=5] {011};
  \node[main node] (7) [below right of=3] {111};

  \path[every node/.style={font=\sffamily\small}]
    (0) edge node [left]      {} (1)
        edge [loop left] node {} (0)
    (1) edge node [left]      {} (3)
        edge node [right]     {} (2)
    (2) edge [bend right] node{} (5)
        edge node [right]     {} (4)
    (3) edge node [right]     {} (6)
        edge node [right]     {} (7)
    (4) edge node [left]      {} (0)
        edge node [right]     {} (1)
    (5) edge [bend right] node{} (2)
        edge node [right]     {} (3)
    (6) edge node [right]     {} (5)
        edge node [right]     {} (4)
    (7) edge [loop right] node{} (7)
        edge node [right]     {} (6);

\end{tikzpicture}

\end{center}
\caption{A directed resolving set in the graph $\vec{\mathcal{B}}(2,3)$ (black vertices).}  \label{exResSet}
\end{figure}

\begin{thm} The directed metric dimension for $\B(d,n)$ is $d^{n-1}(d-1)$.  \end{thm}\label{thm:dirmetdim}

\begin{proof} The following shows that for each $w\in {\cal A}^{n-1}$ a directed resolving set for $\B(d,n)$ must contain (at least) all but one of the vertices with prefix $w$. Suppose that $w\in \A^{n-1}$, and $i\ne j \in \A$ so that neither of $w\oplus i, w\oplus j$ is in our set $S$.  Note that if $x,y\in V(\B(d,n))$, with $x\ne y$, then the distance from $x$ to $y$ is completely determined by $x^-$ (and $y^+$).  Since neither $w\oplus i$ nor $w\oplus j$ is in $S$, and both have the same prefix,  $d(w\oplus i, x)=d(w\oplus j, x)$ for all $x\in S$.  Thus $S$ is not a directed resolving set. Thus for every $w\in \A^{n-1}$, $S$ must contain (at least) all but one of the strings $w\oplus j$ for $j\in \A$.  Thus $|S|\geq d^{n-1}(d-1)$.  Since $\{w\oplus 0 \ | \ w\in \A^{n-1}\}$ can easily be shown to be a directed resolving set, we have the desired equality. \end{proof}

The combination of the previous theorem and Theorem  \ref{DirIC} yields:

\begin{cor}  The directed metric dimension for $\B(d,n)$ is equal to the minimum size of a $t$-identifying code for $\B(d,n)$ if $2t \leq n$. \end{cor}

%

\subsection{Determining Sets}\label{Determining}

In this section we will use a determining set to help us illustrate the automorphism group of $\B(d,2)$, study the relationship between $\aut(\B(d,n-1))$ and $\aut(\B(d,n))$ and use the result to find the determining number for each $\B(d,n)$. First let's recall some definitions.

\begin{defn}
    An \textbf{automorphism} of a graph $G$ is a permutation $\pi$ of the vertex set such that for all pairs of vertices $v,w \in V(G)$, $vw$ is an edge between $v$ and $w$ if and only if $\pi(v)\pi(w)$ is an edge between $\pi(v)$ and $\pi(w)$.  An \textbf{automorphism} of a \textit{directed} graph $G$ is a permutation $\pi$ of the vertex set such that for all pairs of vertices $v,w \in V(G)$, $vw$ is an edge from $v$ to $w$ if and only if $\pi(v)\pi(w)$ is an edge from $\pi(v)$ to $\pi(w)$. One automorphism in the binary (directed or undirected) de~Bruijn graph is a map that sends each string to its complement.
\end{defn}

\begin{defn} \cite{B1}
    A \textbf{determining set}  for $G$ is a set $S$ of vertices of $G$ with the property that the only automorphism that fixes  $S$ pointwise is the trivial automorphism.  The \textbf{determining number} of $G$, denoted $\Det(G)$  is the minimum size of a determining set for $G$.  See Figure \ref{exDetSet} for an example.
\end{defn}

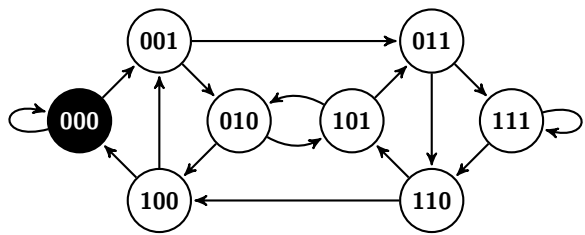
\begin{figure}
\begin{center}

\begin{tikzpicture}[->,>=stealth',shorten >=1pt,auto,node distance=2cm,
  thick,main node/.style={circle,draw,font=\sffamily\bfseries,scale=0.75},new node/.style={circle,fill=black,text=white,draw,font=\sffamily\bfseries,scale=0.75}]

  \node[new node] (0) {000};
  \node[main node]  (1) [above right of=0] {001};
  \node[main node] (2) [below right of=1] {010};
  \node[main node]  (4) [below right of=0] {100};
  \node[main node] (5) [right of=2]       {101};
  \node[main node]  (6) [below right of=5] {110};
  \node[main node]  (3) [above right of=5] {011};
  \node[main node] (7) [below right of=3] {111};

  \path[every node/.style={font=\sffamily\small}]
    (0) edge node [left]      {} (1)
        edge [loop left] node {} (0)
    (1) edge node [left]      {} (3)
        edge node [right]     {} (2)
    (2) edge [bend right] node{} (5)
        edge node [right]     {} (4)
    (3) edge node [right]     {} (6)
        edge node [right]     {} (7)
    (4) edge node [left]      {} (0)
        edge node [right]     {} (1)
    (5) edge [bend right] node{} (2)
        edge node [right]     {} (3)
    (6) edge node [right]     {} (5)
        edge node [right]     {} (4)
    (7) edge [loop right] node{} (7)
        edge node [right]     {} (6);

\end{tikzpicture}

\end{center}
\caption{A minimum size determining set for $\vec{\mathcal{B}}(2,3)$ (black vertex).}  \label{exDetSet}
\end{figure}

Note that an alternate definition for a determining set is a set $S$ with the property that whenever $f,g\in {\rm Aut}(G)$ so that $f(s)=g(s)$ for all $s\in S$, then $f(v)=g(v)$ for all $v\in V(G)$.  That is, every automorphism is completely determined by its action on a determining set.\medskip

Notice that since for both directed resolving sets and for identifying codes, since each vertex in a graph is {\em uniquely} identified by its relationship to the subset by properties preserved by automorphisms, the subset it also a determining set.  Thus every directed resolving set and every identifying code is a determining set.  However, though domination is preserved by automorphisms, vertices are not necessarily uniquely identifiable by their relationship to a dominating set.  Thus a dominating set is not necessarily a determining set.   However, the relationships above mean that the size of a minimum determining set must be at most the size of a minimum identifying code or the directed metric dimension.  For de~Bruijn graphs we have shown that the latter numbers are rather large.  Does this mean that the determining number is also large?  We will see in Corollary \ref{cor:detnum} that the answer for directed de~Bruijn graphs is a resounding \lq No'.

\begin{lem} $S=\{00, 11, 22, 33, \ldots, (d-1)(d-1)\}$ is a determining set for $\vec{\mathcal{B}}(d,2)$.  \end{lem}

\begin{proof}  Suppose that $\sigma \in \aut(\B(d,2))$ fixes $S$ pointwise.  That is, $\sigma(ii) = ii$ for all $i\in \A$.   Choose $ij\ne rs \in V(\B(d,2))$.  Then either $i\ne r$ or $j\ne s$ (or both).  If $i\ne r$ then $\vec d(ii,rs) = 2$ which is distinct from $\vec d(ii, ij)=1$. Since an automorphism of a directed graph must preserve directed distance, $\sigma(ij)\ne rs$ if $i\ne r$. If $j\ne s$, then $\vec d(rs,jj) =2$ which is distinct from $\vec d(ij, jj)=1$.  Thus, again using that $\sigma$ preserves directed distance, $\sigma(ij)\ne rs$ if $j\ne s$.  Thus, $\sigma(ij) = ij$ for all $ij\in V(\B(d,2))$ and therefore $\sigma$ is the identity map and $S$ is a determining set.\end{proof}

Note that we are using directed distances both from and to elements of the set $S$.  Thus $S$ does not fit the definition of a directed resolving set  for $\vec{\mathcal{B}}(d,2)$ (by \cite{DirResolving}, this would require that each vertex $v\in V(G)$ be distinguished by it directed distance  to the vertices of the resolving set).  However directed distances both to and from a set can be used in determining automorphisms of a directed graph.

For the following results, we recall from group theory that $\sym (X)$ is the symmetric group on a set $X$.

\begin{lem}\label{autBd2}  $\aut(\vec{\mathcal{B}}(d,2)) \cong \sym(\A_d)$.\end{lem}

\begin{proof} Let $\sigma\in \sym(\A_d)$.  Define $\varphi_\sigma$ on $V(\vec{\mathcal{B}}(d,n))$ by applying $\sigma$ to each vertex coordinate-wise. That is $\varphi_\sigma(ab) =\sigma(a)\sigma(b)$.  It is easy to show that $\varphi_\sigma$ preserves directed edges and thus is an automorphism.   Further, distinct permutations in $\sym(\A_d)$ produce distinct automorphisms since they act differently on the vertices of the determining set $S$ defined above.  Thus we have an injection $\sym(\A_d) \hookrightarrow \aut(\vec{\mathcal{B}}(d,2))$.\medskip

Since the vertices of $S$ are precisely the vertices with loops, every automorphism of $\B(d,2)$, must preserve $S$ setwise.  This provides the necessary injection from $\aut(\B(d,2)) \hookrightarrow \sym(\A)$.  Thus,  $\aut(\vec{\mathcal{B}}(d,2)) \cong \sym(\A_d)$.\end{proof}

Note that we can consider the automorphisms of $\B(d,2)$ as permutations of the loops, but we can simultaneously consider them as permutations of the symbols in the alphabet $\A_d$. It can be useful to view the automorphisms in these two different ways.\medskip

Additionally, as shown in \cite{Araki}, $\vec{\mathcal{B}}(d,n)$ can be built inductively from $\vec{\mathcal{B}}(d,n-1)$ in the following way.  The vertex $v_1\ldots v_n\in V(\vec{\mathcal{B}}(d,n))$ corresponds to the directed edge from $v_1\ldots v_{n-1}$ to $v_2\ldots v_n$ in $\B(d,n-1)$.  The directed edge $\vec{\mathcal{B}}(d,n)$ from $v_1\ldots v_n \to v_2 \ldots v_n v_{n+1}$  corresponds to the directed 2-path $v_1\ldots v_{n-1} \to v_2\ldots v_n \to v_3\ldots v_{n+1}$ in $\B(d,n-1)$.  That is, $\vec{\mathcal{B}}(d,n)$ is the directed line graph of the directed graph $\vec{\mathcal{B}}(d,n-1)$. Thus, by \cite{Handbook} (Chapter 27, Section 1.1),  $\aut(\B(d,n-1)) = \aut(\B(d,n))$ $\cong$ $\sym(\A)$.  In the following paragraphs we see detail this correspondence.\medskip

Suppose that $\varphi\in \aut(\vec{\mathcal{B}}(d,n))$.  Since $\varphi$ preserves directed edges, we know that both $\varphi(v_1\ldots v_n)=a_1\ldots a_n$
and $\varphi(v_2\ldots v_{n+1}) = b_1\ldots b_n$ if and only if $a_2=b_1, \ldots, a_{n}=b_{n-1}$.  Thus if $\varphi(v_1\ldots v_{n-1} v_n) = a_1\ldots a_{n-1} a_n$ then for every $b\in \A$, $\varphi(v_1\ldots v_{n-1} z) = a_1\ldots a_{n-1} c$ for some $c\in \A$.  In particular, this allow us to define an automorphism $\varphi'\in \aut(\B(d,n-1))$ corresponding to $\varphi\in \aut(\B(d,n))$.  Define $\varphi'$ by $\varphi'(v_1\ldots v_{n-1}) = a_1\ldots a_{n-1}$ where $\varphi(v_1\ldots v_n) = a_1\ldots a_{n-1}$.  By the preceding discussion, $\varphi'$ is well-defined.  It is also clearly a bijection on vertices of $\B(d, n-1)$.  Consider $v_1\ldots v_{n-1}$ and $v_2\ldots v_{n-1}v_n$, the initial and terminal vertices of a directed edge in $\B(d,n-1)$. Since $\varphi$ preserves directed edges  if $\varphi(v_1\ldots v_n) = a_1\ldots a_{n-1} a_n$ then for any $p\in \A$,  $\varphi(v_2\ldots v_n p) = a_2\ldots a_{n} q$ for some $q\in \A$. By definition of $\varphi'$,$ \varphi'(v_1,\ldots v_{n-1}) = a_1\ldots a_{n-1}$ and $\varphi'(v_2\ldots v_n) = a_2\ldots a_n$.  Thus $\varphi'$ preserves the directed edge. Thus we get $\aut(\vec{\mathcal{B}}(d,n))\hookrightarrow \aut(\vec{\mathcal{B}}(d,n-1))$.\medskip

In the other direction, suppose we are given $\varphi'\in \aut(\vec{\mathcal{B}}(d,n-1))$.  Since $\varphi'$ preserves directed edges, and directed edges of $\B(d, n-1)$ are precisely the vertices of $\B(d,n)$, $\varphi'$ defines a map on vertices of $\B(d, n)$.  That is, (with some abuse of notation)
\begin{eqnarray*}
\varphi(v_1\ldots v_n) &=& \varphi(v_1\ldots v_{n-1} \to  v_2\ldots v_n) \\
    &=&  \varphi'(v_1\ldots v_{n-1} \to v_2\ldots v_n) \\
    &=&  \varphi'(v_1\ldots v_{n-1}) \to \varphi'(v_2\ldots v_n).
 \end{eqnarray*}
Thus, given $\varphi'(v_1\ldots v_{n-1})=a_1\ldots a_{n-1}$ then $\varphi'(v_2\ldots v_n) = a_2\ldots a_n$ for some $a_n\in \A$ and we define $\varphi(v_1\ldots v_n)$ $=$ $a_1\ldots a_n$.  Further, since $\varphi'$ preserves directed 2-paths, $\varphi$ preserves directed edges. Thus we get $$\aut(\vec{\mathcal{B}}(d,n-1)) \hookrightarrow \aut(\vec{\mathcal{B}}(d,n)).$$

Since the automorphisms of $\B(d,2)$ are permutations of the loops, and of the symbols of $\A$, by induction, so are the automorphisms of $\B(d,n)$ for all $n$.  Thus we have proved the following.

\begin{thm}  $\aut(\vec{\mathcal{B}}(d,n)) \cong \sym(\A_d)$ for all $n\geq 2$. \end{thm}

\begin{cor}\label{cor:detnum} $\Det(\B(d,n))= \left\lceil \frac{d-1}{n}\right\rceil$. \end{cor}

\begin{proof} Let $S$ be a minimum set of vertices in which each letter of $\A_{d-1}$ occurs at least once. It is easy to see that $|S|=\left\lceil \frac{d-1}{n}\right\rceil$. Any permutation of $\A_d$ that acts nontrivially on any letter of $\A_d$ must act non-trivially on any string containing that letter.  We recall from group theory the following definition: $$\ptstab(S)=\{\sigma \in \sym (\A_d^n) \mid \forall w \in S, \sigma (w) = w\}.$$  Thus if $\sigma\in \ptstab(S)$, then $\sigma$ must fix every letter contained in any string in $S$.  Thus $\sigma$ fixes $0,1,\ldots, d-1$ and therefore also $d$.  We can conclude that $\sigma$  is the identity in both $\sym(\A_d)$ and in $\aut(\B(d,n)$ and therefore $S$ is a determining set.  Thus $\Det(\B(d,n))\leq \left\lceil \frac{d-1}{n}\right\rceil$. \medskip

Further if $|S|<\left\lceil \frac{d-1}{n}\right\rceil$ then fewer than $d-1$ letters of $\A_d$ are used in strings in $S$.  If $a,b\in A_d$ are not represented in $S$, then the transposition $(a \ b)$ in $\sym(A_d)$ is a non-trivial automorphism of $\B(d,n)$ that fixes $S$ pointwise.  Thus $S$ is not a determining set. \end{proof}


Thus for directed de~Bruijn graphs, the determining number and the directed metric dimension can be vastly different in size.

\section{Future Work}

There are several directions that future work in this research area could take.  The first is to continue the research on identifying codes in directed de~Bruijn graphs.  One key result missing from this paper is the determination of the size of a minimum $t$-identifying code in the graph $\vec{\mathcal{B}}(d,n)$ when $t=2n-1$.  It would be ideal to determine both a formula for which graphs $\vec{\mathcal{B}}(d,n)$ are $t$-identifiable for a given $t$, as well as constructions when it is known that such an identifying code exists.

An alternative direction for future research is to consider these same vertex subsets (identifying codes, dominating sets, resolving sets, and determining sets) on the undirected de~Bruijn graph.  Little work has been done and even foundational results like the size of a minimum dominating set are currently missing.  Basic Matlab programs have shown that many more undirected de~Bruijn graphs are $t$-identifiable than directed, and also that the minimum size of an identifying code is much smaller.  For example, through brute force testing we have determined that the minimum size of a 1-identifying code in $\mathcal{B}(2,5)$ is 12, whereas in the directed graph $\vec{\mathcal{B}}(2,5)$ we have shown that the minimum size is 16.

Finally, variations on the concept of identifying code would be useful for real-world applications.  For example, one type of variation known as a $k$-robust identifying code allows for up to $k$ sensor (identifying code vertex) failures without disruption of the identifying code properties.

\bibliographystyle{plain}
\bibliography{DirIDCodes}

\end{document}